\documentclass[12pt]{article}
\usepackage[centertags]{amsmath}
\usepackage{amsfonts}
\usepackage{amssymb}
\usepackage{amsthm}
\DeclareMathAlphabet{\mathpzc}{OT1}{pzc}{m}{it}
\theoremstyle{plain}
\newtheorem{thm}{Theorem}[section]
\newtheorem{cor}[thm]{Corollary}
\newtheorem{lem}[thm]{Lemma}

\theoremstyle{definition}
\theoremstyle{remark}
\numberwithin{equation}{section}

\newcommand{\beast}{\begin{eqnarray*}}
\newcommand{\eeast}{\end{eqnarray*}}

\title{Degrees of irreducible polynomials over\\ binary field}

\author{Yaotsu Chang \footnote{Department of Financial and Computational Mathematics, I-Shou University, Taiwan R.O.C..}
\and Chong-Dao Lee\footnote{Department of Communication Engineering, I-Shou University, Taiwan R.O.C..}
\and Chia-an Liu\footnote{Corresponding author. E-mail address: liuciaan8@gmail.com, Department of Financial and Computational Mathematics, I-Shou University, Taiwan R.O.C..}}

\begin{document}
\maketitle

\bibliographystyle{plain}

\bigskip

\begin{abstract}
An algorithm for factoring polynomials over finite fields is given by Berlekamp in 1967. The main tool was the matrix $Q$ corresponding to each polynomial. This paper studies the degrees of polynomials over binary field that associated with their corresponding matrices $Q$ and irreducibility.
\end{abstract}

\section{Introduction}      \label{sec_introduction}
Let $F=\mathbb{F}_2=\{0,1\}$ be the binary finite field. Then for each polynomial $f(x)\in F[x]$ of degree $m,$ the \emph{Berlekamp matrix} $Q$ proposed in~\cite{b:67} of $f(x)$ is the $m\times m$ matrix over $F$ whose $i$th row represents $x^{2(i-1)}$ reduced modulo $f(x).$ Specifically,
\begin{equation}
x^{2i} \equiv \sum_{j=0}^{m-1}Q_{i+1,j+1}x^j~~~~~~(\text{mod}~f)
\nonumber
\end{equation}
for $i=0,1,\ldots,m-1.$

\medskip

\bigskip

\section{Preliminary}       \label{sec_preliminary}
Let $G$ be the Berlekamp matrix with respect to the polynomial $f(x)$ of degree $m$ over $F.$
It is not difficult to show that $G^m=I_m$ if and only if $f(x)$ has no square factors, where $I_m$ denotes the identity matrix of order $m.$
If $f(x)$ has no square factors then the \emph{order} $o(f(x))=o(G)$ of $f(x)$ is defined to be the least positive integer such that
\begin{equation}
G^{o(G)}=I_m.
\nonumber
\end{equation}
\begin{lem}     \label{lem_order}
Let $f(x)=\prod_{i=1}^{r} g_i(x)$ where $g_i(x)$ are distinct polynomials of order $d_i$ over $F$ for $1\leq i\leq r.$ Then
\begin{equation}
o(G)=\text{lcm}(d_1,d_2,\ldots,d_r).
\nonumber
\end{equation}
\end{lem}

\medskip

It is not hard to have the following observation for least common multiple.
\begin{lem}     \label{lem_lcm}
For positive integers $n_1,n_2,\ldots,n_\ell$ and $k,$ the least common multiple
\begin{equation}
\text{lcm}(kn_1,kn_2,\ldots,kn_\ell)=k\cdot\text{lcm}(n_1,n_2,\ldots,n_\ell).
\nonumber
\end{equation}
\end{lem}
\begin{proof}
Let $A=\text{lcm}(kn_1,kn_2,\ldots,kn_\ell)$ and $B=\text{lcm}(n_1,n_2,\ldots,n_\ell).$ For each $i=1,2,\ldots,\ell,$ one has $kn_i|kB$ since $n_i|B,$ and hence $A|kB.$ On the other hand, $kn_i|A$ implies $n_i|A/k,$ and thus $B|A/k.$ The result follows.
\end{proof}

\bigskip

\section{Main results}      \label{sec_main_results}
Let $G$ be the Berlekamp matrix with respect to the polynomial $f(x)$ of degree $m$ over $F.$
The property $\mathcal{P}_1$ is defined as
\begin{equation}
\mathcal{P}_1:~G^m=I_m~~~\text{if and only if}~~~f(x)~\text{is irreducible}
\nonumber
\end{equation}
where $I_m$ is the identity matrix of order $m.$

\medskip

\begin{thm}
Let $f(x)$ be a polynomial over $F$ of degree $m\geq 2.$ Then $f(x)$ has the property $\mathcal{P}_1$ if and only if
\begin{equation}
m~\text{is an odd prime or 9}.
\nonumber
\end{equation}
\end{thm}
\begin{proof}
Note that $G^m=I_m$ if and only if $f(x)$ divides $x^{2^m}-x.$ Then $\mathcal{P}_1$ is also realized as
\begin{equation}
\mathcal{P}_1':~f(x)~\text{divides}~x^{2^m}-x~\text{if and only if}~f(x)~\text{is irreducible.}
\nonumber
\end{equation}
Moreover, since
\begin{equation}
x^{2^m}-x=\prod_{d\mid m}(\text{Irreducible polynomials of degree}~d),
\nonumber
\end{equation}
the property $\mathcal{P}_1'$ is equivalent to
\begin{equation}
\mathcal{P}_1'':~\text{If}~f(x)~\text{divides}~x^{2^m}-x~\text{then}~f(x)~\text{is irreducible.}
\nonumber
\end{equation}

To prove the sufficiency, suppose that $f(x)$ has the property $\mathcal{P}_1''.$
Note that the number of irreducible polynomials $N(\ell)$ over $F$ of degree $\ell$ is
\begin{equation}
N(\ell)\left\{\begin{array}{ll}
                =2 & \text{if}~\ell=1   \\
                 =\ell-1 & \text{if}~\ell=2,3,4 \\
                 \geq\ell & \text{if}~\ell>4
               \end{array}
\right..
\nonumber
\end{equation}
Hence the degree $m$ of $f(x)$ can not be written as $m=n\ell$ for some positive integers $1<n\leq\ell$ and $\ell\geq 5,$ or otherwise $N(\ell)\geq n$ and a product of $n$ irreducible polynomials of degree $\ell$ does not have $\mathcal{P}_1''.$ For the some reason, $m$ can not be written as $m=n\ell$ for some positive integers $1<n<\ell,$ either. Furthermore, $m\neq 2$ and $m\neq 4$ since both of $x(x+1)$ and $x(x+1)(x^2+x+1)$ do not have $\mathcal{P}_1''.$
To conclude the above argument, $m$ is an odd prime or $m=3\cdot 3=9.$

For the necessity, assume $m$ is an odd prime or $9.$ If $m$ is an odd prime, then
$$x^{2^m}-x=x(x+1)\prod\text{(Irreducible polynomials of degree $m$)},$$
and hence $f(x)$ is an irreducible polynomial of degree $m$ whenever $f(x)$ divides $x^{2^m}-x.$ Besides, if $m=9$ then
$$x^{2^9}-x=x(x+1)(x^3+x+1)(x^3+x^2+1)\prod\text{(Irreducible polynomials of degree $9$)},$$
and thus $f(x)$ is an irreducible polynomial of degree $9$ provided that $f(x)$ divides $x^{2^9}-x.$
It says that $f(x)$ has $\mathcal{P}_1'',$ and the proof is completed.
\end{proof}

\medskip

Let $G$ be the Berlekamp matrix with respect to the polynomial $f(x)$ of degree $m$ over $F.$ The property $\mathcal{P}_2$ is defined as
\begin{equation}
\mathcal{P}_2:~o(G)=m~~~\text{if and only if}~~~f(x)~\text{is irreducible}
\nonumber
\end{equation}
where $o(G)$ is the order of $G.$

\medskip

\begin{thm}
Let $f(x)$ be a polynomial over $F$ of degree $m\geq 2.$ If $f(x)$ has the property $\mathcal{P}_2$ then $m$ can be written as
\begin{equation}
m=p^i~~~~~~\text{or}~~~~~~m=p^iq
\nonumber
\end{equation}
for primes $p<q$ and positive integer $i.$
\end{thm}
\begin{proof}
Suppose to the contrary that $m$ can be written as a product of three pairwise coprime factors that more than $1,$ or $m=p^iq^j$ for primes $p<q$ and positive integers $i,j$ with $j\geq 2.$

Assume that $m$ can be written as $m=p_1p_2p_3,$ where positive integers $1<p_1<p_2<p_3$ are pairwise coprime. Since
\begin{equation}
p_1p_2p_3=1\cdot p_1p_3+(p_2-p_1)\cdot p_2+(p_1-1)\cdot p_2p_3,
\nonumber
\end{equation}
a product of $1,$ $p_2-p_1,$ and $p_1-1$ irreducible polynomials respectively of degrees $p_1p_3,$ $p_2,$ and $p_2p_3$ does not have the property $\mathcal{P}_2,$ which is a contradiction.
(It is quick to check that the number $N_2(\ell)$ of irreducible polynomials over $F$ of degree $\ell$ satisfies $N_2(\ell)\geq \ell-1$ for each positive integer $\ell,$ and $\text{lcm}(p_1p_3,p_2,p_2p_3)=p_1p_2p_3=m.$)

Next, suppose that $m$ can be written as $m=p^iq^j$ for primes $p<q$ and positive integers $i,j$ with $j\geq 2.$ Since
\begin{equation}
p^iq^j=(p-1)\cdot p^{i-1}q^j+1\cdot p^iq^{j-1}+(q-p)\cdot p^{i-1}q^{j-1},
\nonumber
\end{equation}
a product of $p-1,$ $1,$ and $q-p$ irreducible polynomials respectively of degrees $p^{i-1}q^j,$ $p^iq^{j-1},$ and $p^{i-1}q^{j-1}$ does not have the property $\mathcal{P}_2,$ which is a contradiction.
(It is immediate to check that $\text{lcm}(p^{i-1}q^j,p^iq^{j-1},p^{i-1}q^{j-1})=p^iq^j=m.$)
The result follows.
\end{proof}

\medskip

The case will be trivial if the degree $m$ is a prime power. It may be quick to show that $f(x)\in F[x]$ has $\mathcal{P}_2$ if its degree $m=p^i$ for some prime $p$ and positive integer $i.$
Now, focus on the case $m=p^iq$ for primes $p<q$ and positive integer $i.$

\medskip

Note that the order $o(f(x))$ of $f(x)$ is the least common multiple (l.c.m.) of the degrees of factors in $f(x).$
Then a quick observation is given below.
\begin{lem}     \label{lem_p^i}
Assume that $f(x)\in F[x]$ is of degree $m=p^iq$ for primes $p<q$ and positive integer $i.$ If the order $o(f(x))=m$ then there exists a factor of $f(x)$ that of degree $p^i$ or $p^iq.$
\end{lem}
\qed

\medskip

\begin{cor}
Let $f(x)\in F[x]$ of degree $m=p^iq$ for primes $p<q$ and positive integer $i.$ Then the following properties follow.
\begin{itemize}
\item[(i)]      If $f(x)$ has $\mathcal{P}_2$ then a polynomial $\tilde{f}(x)\in F[x]$ of degree $\tilde{m}=p^{\tilde{i}}q$ with positive integer $\tilde{i}\leq i$ also has $\mathcal{P}_2.$
\item[(ii)]     If $q>2^{p^i}$ then $f(x)$ has $\mathcal{P}_2.$
\item[(iii)]    If $p^i=2$ then $q>4$ if and only if $f(x)$ has $\mathcal{P}_2.$
\item[(iv)]     If $q>p^i>2$ and $f(x)$ has $\mathcal{P}_2,$ then $(p^i-2)q>2^{p^i}-2^{p^{i-1}}+1.$ If $q<p^i$ and $f(x)$ has $\mathcal{P}_2,$ then $(q-2)p^i>2^q.$
\end{itemize}
\end{cor}
\begin{proof}
(i) is direct from Lemma~\ref{lem_lcm}.

\smallskip

To prove (ii), suppose $q>2^{p^i}$ and the order $o(f(x))$ of $f(x)$ equals $m=p^iq.$ Since $f(x)$ is of degree $m$ and divides
\begin{equation}
x^{2^m}-x=(x^{2^{p^i}}-x)\prod_{j=1}^i\left(\text{Irreducible polynomials of degree}~qj\right),
\nonumber
\end{equation}
the degree of each factor of $f(x)$ is a multiple of $q.$ Hence by Lemma~\ref{lem_p^i} there is a factor in $f(x)$ that of degree $p^iq,$ which means that $f(x)$ is composed of exactly one factor polynomial and is irreducible.

\smallskip

The sufficiency of (iii) is straightforward from (ii). Then the necessity of (iii) is only to examine the case $q=3,$ \emph{i.e.}, $m=2\cdot 3=6.$ One can see that the reducible polynomial $f(x)=x(x^2+x+1)(x^3+x+1)$ is of order $o(f(x))=6,$ and thus does not have $\mathcal{P}_2.$ The proof of (iii) is completed.

\smallskip

To prove the first part of (iv), suppose to the contrary that $q>p^i>2$ and $f(x)$ has $\mathcal{P}_2,$ but $(p^i-2)q\leq2^{p^i}-2^{p^{i-1}}+1.$ Since $p<q$ are two distinct primes, one has $\gcd(p^i,q)=1,$ and thus
\begin{equation}
\{uq~(\text{mod}~p^i)~\mid~u=1,2,\ldots,p^i-1\}=\{1,2,\ldots,p^i-1\}.
\nonumber
\end{equation}
By Pigeon Hole Principle, there exists positive integer $\hat{u}$ with $1\leq u\leq p^i-2$ such that
$\hat{u}q\equiv 1~\text{or}~2~(\text{mod}~p^i).$ If $\hat{u}q\equiv 1~(\text{mod}~p^i),$ say, $\hat{u}q=\ell p^i+1,$ then the product of $x$ and $\ell$ irreducible polynomials of degree $p^i$ and $p^i-\hat{u}$ irreducible polynomials of degree $q$ does not has $\mathcal{P}_2,$ which is a contradiction. (Note that the fact $N(p^i)=(2^{p^i}-2^{p^{i-1}})/p^i$ and $(p^i-2)q\leq2^{p^i}-2^{p^{i-1}}+1$ implies $\ell$ exists.) On the other hand, if $\hat{u}q\equiv 2~(\text{mod}~p^i),$ say, $\hat{u}q=\ell p^i+2,$ then the product of $x(x+1)$ and $\ell$ irreducible polynomials of degree $p^i$ and $p^i-\hat{u}$ irreducible polynomials of degree $q$ does not has $\mathcal{P}_2,$ which makes a contradiction.
Similarly, to show the second part of (iv), suppose to the contrary that $q<p^i$ and $f(x)$ has $\mathcal{P}_2,$ but $(q-2)p^i\leq2^q.$ An analogue version of the above contradiction will occur by exchanging the positions of $p^i$ and $q.$ The result follows.
\end{proof}

\medskip

%

\end{document}